\documentclass[12pt, reqno]{amsart}
\usepackage{amsmath, amsthm, amscd, amsfonts, amssymb, graphicx, color}
\usepackage[bookmarksnumbered, colorlinks, plainpages]{hyperref}
\hypersetup{colorlinks=true,linkcolor=red, anchorcolor=green, citecolor=cyan, urlcolor=red, filecolor=magenta, pdftoolbar=true}

\newtheorem{theorem}{Theorem}[section]
\newtheorem{lemma}[theorem]{Lemma}
\newtheorem{proposition}[theorem]{Proposition}
\newtheorem{corollary}[theorem]{Corollary}
\theoremstyle{definition}

\newtheorem{example}[theorem]{Example}

\theoremstyle{remark}

\numberwithin{equation}{section}

\begin{document}
\title[Gr\"{u}ss type inequalities for positive linear maps on $C^*$-algebras]{Gr\"{u}ss type inequalities for positive linear maps on $C^*$-algebras}
\author[A. Dadkhah and M.S. Moslehian ]{Ali Dadkhah and Mohammad Sal Moslehian}
\address{Department of Pure Mathematics, Center Of Excellence in Analysis on Algebraic Structures (CEAAS), Ferdowsi University of Mashhad, P. O. Box 1159, Mashhad 91775, Iran}
\email{dadkhah61@yahoo.com}
\email{moslehian@um.ac.ir}
\begin{abstract}
Let $\mathcal{A}$ and $\mathcal{B}$ be two unital $C^*$-algebras and let for $C\in\mathcal{A},\ \Gamma_C=\{\gamma \in \mathbb{C} : \|C-\gamma I\|=\inf_{\alpha\in \mathbb{C}} \|C-\alpha I\|\}$. We prove that if $\Phi :\mathcal{A} \longrightarrow \mathcal{B}$ is a unital positive linear map, then
\begin{eqnarray*}
\big|\Phi(AB)-\Phi(A)\Phi(B)\big| \leq \big\|\Phi(|A^*-\zeta I|^2)\big\|^\frac{1}{2} \big[\Phi(|B-\xi I|^2)\big]^\frac{1}{2}
\end{eqnarray*}
for all $A,B\in\mathcal{A}, \zeta \in \Gamma_A$ and $\xi\in\Gamma_B.$\\
In addition, we show that if $(\mathcal{A},\tau)$ is a noncommutative probability space and $T \in \mathcal{A}$ is a density operator, then
\begin{eqnarray*}
\ \ \big|\tau(TAB)-\tau(TA)\tau(TB)\big|\leq \|A-\zeta I\|_p\|B-\xi I\|_q\|T\|_r \ \ (p,q\geq 4, r\geq 2)
\end{eqnarray*}
and
\begin{eqnarray*}
\big|\tau(TAB)-\tau(TA)\tau(TB)\big|\leq \|A-\zeta I\|_p\|B-\xi I\|_q\|T\| \ \ \ \ (p,q\geq 2)\ \ \ \ \
\end{eqnarray*}
for every $A,B \in \mathcal{A}$ and $\zeta \in \Gamma_A,\xi \in \Gamma_B$. Our results generalize the corresponding results  for matrices to operators on spaces of arbitrary dimension.
\end{abstract} \maketitle

%-------------------------------------------------------------------------------
\section{introduction}

Gr\"{u}ss \cite{grus} showed that if $f$ and $g$ are integrable real functions on $[a, b]$ and there exist real constants $\alpha,\beta,\gamma,\Gamma$ such that $\alpha\leq f (x)\leq \beta$ and
$\gamma\leq g(x)\leq \Gamma$ for all $x\in [a, b]$, then
\begin{eqnarray*}
\left|\dfrac{1}{b-a}\int_a^b f(x)g(x)dx-\dfrac{1}{b-a}\int_a^b f(x)dx\ \dfrac{1}{b-a}\int_a^b g(x)dx\right|\leq \frac{1}{4}|\beta-\alpha| |\Gamma -\gamma|.
\end{eqnarray*}
This inequality was studied and extended by a number of mathematicians for different contents, such as inner product spaces,
quadrature formulae, finite Fourier transforms and linear functionals. For further information we refer interested reader to \cite{mata,peri} and references therein.\\
\noindent Bhatia and Davis \cite{bh1} showed that if $\Phi$ is any positive unital linear map between $C^*$-algebras and $m\leq A \leq M$ is any self-adjoint operator, then
\begin{eqnarray*}\label{bhatiadavis}
\Phi(A^2)-\Phi(A)^2\leq \frac{1}{4}(M-m)^2.
\end{eqnarray*}
This provides a reverse to the Kadison inequality $\Phi(A)^2 \leq \Phi(A^2)$; cf. \cite{KAD}. The authors of \cite{bh3} extended this inequality by showing that
\begin{eqnarray}\label{bhat}
\Phi(A^*A)-\Phi(A)^* \Phi(A) \leq \inf_{\alpha \in \mathbb{C}}\|A-\alpha I\|^2.
\end{eqnarray}
for every operator $A$. The inequality (\ref{bhat}) turns out to be useful in deriving various lower bounds for the spread by Bhatia
and Sharma \cite{bh2, bh23}. The work is then extended by Sharma and Kumari \cite{kumar} to
discuss the perturbation bounds for matrices. Moreover, the second author and Raji\'{c} \cite{mo2} proved that
\begin{eqnarray}\label{moraj}
\|\Phi(AB)-\Phi(A) \Phi(B) \|\leq \inf_{\alpha \in \mathbb{C}}\|A-\alpha I\| \inf_{\beta \in \mathbb{C}}\|B-\beta I\|
\end{eqnarray}
for any unital $n$-positive linear map $\Phi$ ($n\geq 3$) and any operators $A,B$ in a $C^*$-algebra.
 Matharu and the second author \cite{mata} gave several variants of inequality (\ref{moraj}) for unital completely positive linear maps and unitary invariant norms.
\\
Renaud \cite{renaud} showed that
\begin{eqnarray}\label{rena}
\big|{\rm tr}(TAB) -{\rm tr}(TA){\rm tr}(TB)\big| \leq krs,
\end{eqnarray}
for $1 \leq k\leq 4,$ where $A,B$ are square matrices, whose numerical ranges lie in the circular
discs of radii $r$ and $s$, respectively, and $T$ is a positive matrix of trace one.\\
The Gr\"{u}ss inequality was generalized in the setting of inner product spaces by Dragomir \cite{dra1}, and in the framework of inner product modules over $H^*$-algebras and $C^*$-algebras by Ili\v{s}evi\'{c} et al. \cite{il1,il2}.\\
In this paper, we study some Gr\"{u}ss type inequalities. In section \ref{plmsection}, we obtain two refined bounds for left sides of (\ref{bhat}) and (\ref{moraj}). We then give some applications of these inequalities to finite traces on noncommutative probability spaces. At the end of this section, we present some inequalities under certain mild condition.\\
 In section \ref{modulesection}, we generalize some Gr\"{u}ss type inequalities for Hilbert $C^*$-modules, as well as adjointable operators on these spaces.

\section{Preliminaries}

\noindent
Let $\mathbb{B}(\mathcal{H})$ denote the $C^*$-algebra of all bounded linear operators on a complex Hilbert space $(\mathcal H, \langle\cdot,\cdot\rangle).$ Throughout the paper, a capital letter means an operator in $\mathbb{B}(\mathcal{H})$. An operator $A$ is called positive if $\langle Ax, x\rangle\geq 0 $ for all $x \in \mathcal{H}$, and we then write $A\geq 0$. An operator A is said to be strictly positive (denoted by $ A >0$) if it is a positive invertible operator. The unit of $\mathbb{B}(\mathcal{H})$ is denoted by $I$. According to the Gelfand--Naimark--Segal theorem, every $C^*$-algebra can be regarded as a $C^*$-subalgebra of $\mathbb{B}(\mathcal{H})$ for some Hilbert space $\mathcal{H}$. A von Neumann algebra is a strongly operator closed $*$-subalgebra of $\mathbb{B}(\mathcal{H})$, which contains $I$.\\
In this paper, we use $\mathcal{A},\mathcal{B}, \cdots $ to denote $C^*$-algebras and von Neumann algebras and $\mathcal{A}_+$ to denote the set of all positive elements of $\mathcal{A}$. An operator $A$ is called {accretive} if ${\rm Re} (A)\geq 0$, where ${\rm Re}(A)=(A+A^*)/2.$ \\
 A linear map $\Phi :\mathcal{A}\longrightarrow \mathcal{B}$ between $C^*$-algebras is said to be $*$-linear if $\Phi(A^*)=\Phi(A)^*$. It is positive if $\Phi(A)\geq 0,$ whenever $A \geq 0.$
We say that $\Phi$ is unital if $\mathcal{A}, \mathcal{B}$ are unital
$C^*$-algebras and $\Phi$ preserves the identity. A linear map $\Phi$ is called $n$-positive if the map $\Phi_n : M_n(\mathcal{A})\longrightarrow M_n(\mathcal{B})$ defined by $\Phi_n([a_{ij}]_{n\times n}) = [\Phi(a_{ij})]_{n\times n}$ is positive, where $M_n(\mathcal{A})$ stands for the $C^*$-algebra of
$n \times n$ matrices with entries in $\mathcal{A}$. A map $\Phi$ is said to be completely positive if it is $n$-positive for every $n\in \mathbb{N}$. It is known that every positive functional is completely positive. Moreover, the following Cauchy--Schwarz inequality holds; cf. \cite[Section 3.3]{mur}:
\begin{eqnarray*}
|\tau(A^*B)|\leq \tau(A^*A)^\frac{1}{2} \tau(B^*B)^\frac{1}{2}
\end{eqnarray*} for every positive linear functional $\tau$ on $\mathcal{A}$ and $A,B\in \mathcal{A}$.\\ A finite trace $\tau$ on a von Neumann algebra $\mathcal{A}$ is a positive linear functional on $\mathcal{A}$ such that $\tau(I)<\infty$ and $\tau(AB)=\tau(BA)$ for all $A,B\in \mathcal{A}$. A trace $\tau$ is said to be normal if $\sup_i \tau(A_i) =\tau (\sup_i A_i)$ for any bounded increasing net $(A_i)$
in $\mathcal{A}_+$. It is faithful if $\tau(A) = 0$ implies $A = 0$ and normalized if $\tau (I) =1$. \\ A noncommutative probability space $(A,\tau )$ is a von Neumann algebra $\mathcal{A}$ with a normal faithful normalized finite trace $\tau$. In this case, for $ 1 \leq p<\infty$ and $A\in \mathcal{A}$, $\|A\|_p= (\tau (|A|^p))^\frac{1}{p}$ gives rise to a norm on $\mathcal{A}$. It is known that if $p\leq q$, then $\|A\|_p\leq \|A\|_q \leq\|A\|$
 for every $A,B \in \mathcal{A}$; see \cite[Proposition 1.9 and Remark 1.4]{xu}.\\
The notion of Hilbert $C^*$-module is a natural generalization of that of Hilbert space arising by replacing the field of scalars $\mathbb{C}$ by a $C^*$-algebra. Recall that if $(\mathcal{X},\langle \cdot ,\cdot\rangle) $ is Hilbert $C^*$-module over a $C^*$-algebra $\mathcal{A}$, then for every $x\in \mathcal{X}$ the norm on $\mathcal{X}$ is given by $\|x\|=\|\langle x,x\rangle\|^{\frac{1}{2}}$ and the ``absolute-value norm'' is defined by $|x|=\langle x,x \rangle ^{\frac{1}{2}}$ as a positive element of $\mathcal{A}$. A map $T:\mathcal{X}\longrightarrow \mathcal{X}$ is called adjointable if there is a map $T^*:\mathcal{X}\longrightarrow \mathcal{X}$  such that $\langle Tx,y\rangle=\langle x,T^*y \rangle$ for all $x,y \in \mathcal{X}$. The set of all adjointable maps $T:\mathcal{X}\longrightarrow\mathcal{X}$ is denoted by $\mathcal{L}(\mathcal{X})$, which is a unital $C^*$-algebra. For $x,y \in \mathcal{X}$, the rank one operator $x\otimes y:\mathcal{X}\longrightarrow\mathcal{X}$ is defined by $(x\otimes y) (z)=x\langle y,z\rangle$. An element $h\in \mathcal{X}$ is called a lifted projection if
 $|h|$ is a projection in $\mathcal{A}$. According to \cite{il2}, $h$ is a lifted projection in $\mathcal{X}$ if and only
 if $h$ is non zero and $h|h|=h.$ Clearly $h\otimes h$ is a projection in $\mathcal{L}(\mathcal{X}).$ \\
 For each element $A\in \mathcal{A},$ we define $R_A:\mathcal{X}\longrightarrow\mathcal{X}$ by $R_A x=xA$. It is easy to check that $R_A\in \mathcal{L}(\mathcal{X})$ if and only if $A\in \mathcal{Z}(\mathcal{A})$, where $\mathcal{Z}(\mathcal{A})$ is the center of $\mathcal{A}$. \\
The Cauchy--Schwarz inequality for $x,y\in \mathcal{X}$ asserts that (see \cite{lan})
 \begin{eqnarray*}
 \langle x,y\rangle \langle y,x\rangle\leq \|\langle y,y\rangle \| \langle x,x\rangle .
\end{eqnarray*}

\section{inequalities for positive linear maps}\label{plmsection}

In this section, we refine and generalize some inequalities for positive linear maps. Moreover, we give some Gr\"{u}ss type inequalities for traces on noncommutative probability spaces. We start our work with the following lemma.
\begin{lemma}\label{main1}
Let $\mathcal{A}$ and $\mathcal{B}$ be two unital $C^*$-algebras. If $\Phi :\mathcal{A} \longrightarrow \mathcal{B}$ is a unital $*$-linear map, then
\begin{eqnarray}\label{variancep}
\Phi(A^*A)-\Phi(A)^*\Phi(A) \leq \Phi(|A-\alpha I|^2)
\end{eqnarray}
for all $A\in\mathcal{A}$ and $\alpha \in \mathbb{C}$.

\end{lemma}
\begin{proof}
Since $\Phi$ is a unital $*$-linear map, we have
\begin{eqnarray*}
\Phi(|A-\alpha I|^2 ) &\geq& \Phi \left((A-\alpha I)^* (A-\alpha I)\right)-\left(\Phi (A-\alpha I)\right)^*\Phi (A-\alpha I)\\ &=& \Phi(A^*A)-\alpha \Phi(A^*)-\overline{\alpha}\Phi(A) +\overline{\alpha}\alpha I - \Phi(A)^*\Phi(A)\\ && \ +\alpha \Phi(A^*)+\overline{\alpha}\Phi(A) -\overline{\alpha}\alpha I \\ &=& \Phi(A^*A)-\Phi(A)^*\Phi(A)
\end{eqnarray*}
for all $\alpha \in \mathbb{C}$ and for all $A\in \mathcal{A}$.
\end{proof}

Let $A$ be an element of a unital $C^*$-algebra. It is mentioned in (\cite[Theorem 4]{der}) that there is a scalar $\gamma \in \mathbb{C}$ such that
\begin{eqnarray*}\label{infdelta}
\inf_{\alpha \in \mathbb{C}}\|A-\alpha I\|=\|A-\gamma I\|.
\end{eqnarray*}
Indeed, let $ \Sigma=\{\|A-\alpha I\|:\alpha \in \mathbb{C}\}$ and $\Omega=\{\|A-\beta I\|:|\beta| \leq 3\|A\|\}.$ We see that $\inf \Sigma= \inf\Omega$. In fact, for every $\alpha \in \mathbb{C}$ there is an element $\beta$ with $|\beta|\leq 3\|A\|$ such that $\|A-\alpha I \|\geq \|A-\beta I\|$.
If $|\alpha|> 3\|A\|$, then by choosing $\beta=\|A\|$, we have
\begin{eqnarray*}
\|A-\alpha I\| \geq |\alpha|-\|A\| > 2\|A\|\geq \|A-\beta I\|.
\end{eqnarray*}
If $|\alpha|\leq 3\|A\|$, then we set $\beta=\alpha$.\\
Now, Let $\Gamma_A=\{\gamma \in \mathbb{C} : \|A-\gamma I\|=\inf_{\alpha \in \mathbb{C}} \|A-\alpha\|\}$. If $\Phi$  is a unital positive linear map, then
\begin{equation}\label{variancep2}
\Phi(A^*A)-\Phi(A)^*\Phi(A)\leq \Phi(|A-\alpha I|^2)\leq \Phi(\|A-\alpha I\|^2 I)
=\|A-\alpha I\|^2I
\end{equation}
for every $\alpha \in \mathbb{C}$. Hence
\begin{eqnarray*}
\Phi(A^*A)-\Phi(A)^*\Phi(A)\leq \inf_{\alpha \in \mathbb{C}} \|A-\alpha I\|^2I =\|A-\gamma I\|^2\quad (\gamma \in \Gamma_A).
\end{eqnarray*}

Now we give an example to show that inequality (\ref{variancep2}) is really finer than inequality (\ref{bhat}).
\begin{example}
Let $\Phi :M_2(\mathbb{C})\longrightarrow \mathbb{C}$ be given by $\Phi \left(\begin{bmatrix}
 a_{11}&a_{12}\\ a_{21}&a_{22}
 \end{bmatrix}\right)=
a_{11}.$
 It is evident that $\Phi $ is a unital positive linear map. Take $A= \begin{bmatrix}
1&2 \\ 2&4
\end{bmatrix}$. It is clear that the spectrum of $A$ is $\sigma(A)=\{0,5\}$. Since A is self-adjoint, we have
 $\inf_{\lambda \in \mathbb{C}}\|A-\lambda I_{2\times 2}\|^2 =(2.5)^2=6.25$, where $I_{2 \times 2}$ denotes the identity of the matrix algebra $M_2(\mathbb{C})$; see \cite[Corollary 1]{der}. If we put $\alpha=0$, then
\begin{eqnarray*}
\Phi(A^*A)-\Phi(A)^*\Phi(A)=4,\ \Phi (|A-\alpha I_{2\times 2}|^2) =5 .
 \end{eqnarray*}
Therefore, we see that
\begin{eqnarray*}
\Phi(A^*A)-\Phi(A)^*\Phi(A)<\Phi (|A-\alpha I_{2\times2}|^2)<\inf_{\lambda \in \mathbb{C}}\|A-\lambda I_{2\times 2}\|^2.\end{eqnarray*}
This example shows that even if $\alpha $ is not in $\Gamma_A$, inequality (\ref{variancep2}) can give a finer bound than that in inequality (\ref{bhat}).
\end{example}

Now we want to give a Gr\"{u}ss type inequality for $n$-positive linear maps. To achieve it we need two lemmas.
\begin{lemma}\label{matrixp} \cite[Lemma 2.1]{choi}
Let $A\geq 0,  B> 0$ be two operators in $ M_n(\mathcal{A})$. Then the block matrix $\begin{bmatrix}A& X\\ X^*& B  \end{bmatrix}$ is positive if and only if $A\geq X B^{-1} X^*$.
\end{lemma}
The next lemma is known \cite[Theorem 1]{mathi}. We prove it since we need some portion of the proof later.
\begin{lemma}\label{matrixvar}
Let $\mathcal{A}$ and $\mathcal{B}$ be two unital $C^*$-algebras. If $\Phi :\mathcal{A} \longrightarrow \mathcal{B}$ is a unital $n$-positive linear map for some $n\geq 3,$ then
 \begin{eqnarray*}
\begin{bmatrix}  \Phi(A^*A)-\Phi(A)^*\Phi(A) & \Phi(A^*B)-\Phi(A)^*\Phi(B)  \\ \Phi(B^*A)-\Phi(B)^*\Phi(A)   & \Phi(B^*B)-\Phi(B)^*\Phi(B)
\end{bmatrix} \geq 0
\end{eqnarray*}
for every $A,B\in \mathcal{A}$.
\end{lemma}
\begin{proof}
First we have
 \begin{eqnarray}\label{matrixp2}
\begin{bmatrix}  A^*A & A^*B & A^* \\ B^*A & B^*B & B^* \\ A& B & I
\end{bmatrix} =\begin{bmatrix} A^*&0&0 \\ B^*&0&0 \\ I&0&0  \end{bmatrix}\begin{bmatrix} A&B&I\\ 0&0&0 \\ 0&0&0  \end{bmatrix}\geq 0.
\end{eqnarray}
It follows from $3$-positivity of $\Phi$ that
 \begin{eqnarray*}
\begin{bmatrix}  \Phi(A^*A) & \Phi(A^*B) & \Phi(A^*) \\ \Phi( B^*A) &\Phi( B^*B)& \Phi(B^*) \\ \Phi(A)& \Phi(B) & I
\end{bmatrix} \geq 0,
\end{eqnarray*}
which implies that
 \begin{eqnarray*}
\begin{bmatrix}  \Phi(A^*A) & \Phi(A^*B) & \Phi(A^*) & 0\\ \Phi( B^*A) &\Phi( B^*B)& \Phi(B^*) & 0 \\ \Phi(A)& \Phi(B) & I & 0 \\ 0 & 0 & 0 & I
\end{bmatrix} \geq 0.
\end{eqnarray*}
Hence, by applying Lemma \ref{matrixp}, we assert    that
\begin{eqnarray*}
\begin{bmatrix}  \Phi(A^*A) & \Phi(A^*B) \\ \Phi(B^*A)  & \Phi(B^*B)
\end{bmatrix} \geq \begin{bmatrix}  \Phi(A)^* & 0 \\ \Phi(B)^* & 0\end{bmatrix} \begin{bmatrix} I & 0\\ 0 & I
\end{bmatrix} \begin{bmatrix}  \Phi(A) & \Phi(B) \\ 0  & 0\end{bmatrix}.
\end{eqnarray*}
In other words,
\begin{eqnarray*}
\begin{bmatrix}  \Phi(A^*A) & \Phi(A^*B) \\ \Phi(B^*A)  & \Phi(B^*B)
\end{bmatrix} \geq \begin{bmatrix}  \Phi(A)^*\Phi(A) & \Phi(A)^*\Phi(B) \\ \Phi(B)^*\Phi(A) & \Phi(B)^*\Phi(B)\end{bmatrix}.
\end{eqnarray*}
\end{proof}

The following theorem is a refinement of inequality (\ref{moraj}).
\begin{theorem}\label{grussp}
Let $\mathcal{A}$ and $\mathcal{B}$ be two unital $C^*$-algebras. If $\Phi :\mathcal{A} \longrightarrow \mathcal{B}$ is a unital $n$-positive linear map for some $n\geq 3,$ then
\begin{eqnarray}\label{covariance}
\big|\Phi(AB)-\Phi(A)\Phi(B) \big|\leq \big\|\Phi(|A^*-\alpha I|^2)\big\|^\frac{1}{2} \big[ \Phi(|B-\beta I|^2)\big]^\frac{1}{2}
\end{eqnarray}
for every $A,B\in\mathcal{A}$ and $\alpha,\beta \in \mathbb{C}.$
In particular,
\begin{eqnarray}\label{refine}
\big\|\Phi(AB)-\Phi(A)\Phi(B) \big\|&\leq& \inf_{\alpha \in \mathbb{C}}\|A-\alpha I\|\inf_{\beta \in \mathbb{C}}\|A-\beta I\|.
\end{eqnarray}
\end{theorem}
\begin{proof}
Without loss of generality we can assume that $\Phi(A^*A)-\Phi(A)^*\Phi(A) >0$. According to Lemma \ref{matrixp} and Lemma \ref{matrixvar}, we may state that

\begin{eqnarray*}
 \Phi(B^*B)-\Phi(B)^*\Phi(B) &\geq& \big(\Phi(B^*A)-\Phi(B)^*\Phi(A)\big) \\ && \cdot \big(\Phi(A^*A)-\Phi(A)^*\Phi(A)\big)^{-1} \big(\Phi(A^*B)-\Phi(A)^*\Phi(B)\big)\\ &\geq& \big(\Phi(B^*A)-\Phi(B)^*\Phi(A)\big) \\ && \cdot \big\|\Phi(A^*A)-\Phi(A)^*\Phi(A)\big\|^{-1} \big(\Phi(A^*B)-\Phi(A)^*\Phi(B)\big).
\end{eqnarray*}
Hence, we infer
\begin{eqnarray*}
\big|\Phi(A^*B)-\Phi(A)^*\Phi(B) \big|^2\leq \big\|\Phi(A^*A)-\Phi(A)^*\Phi(A)\big\|\big(\Phi(B^*B)-\Phi(B)^*\Phi(B)\big).
\end{eqnarray*}
Now, by using Lemma \ref{main1} and taking square root, we get
\begin{eqnarray*}
\big|\Phi(A^*B)-\Phi(A)^*\Phi(B) \big|\leq \big\|\Phi(|A-\alpha I|^2)\big\|^\frac{1}{2} \big(\Phi(|B-\beta I|^2)\big)^\frac{1}{2}
\end{eqnarray*}
for all $\alpha,\beta \in \mathbb{C}$.
Replacing $A$ with $A^*$ in the latter inequality we obtain (\ref{covariance}). Next, by using the fact that $\Gamma_{A^*}=\Gamma_A$ and applying inequality (\ref{variancep2}) we can get (\ref{refine}).
\end{proof}
The following example shows that we can't replace the condition $n\geq 3$ in inequality \eqref{covariance} of Theorem \ref{grussp} by $n\geq 1$ or $n\geq 2$.
\begin{example}
Let $\varphi: M_2(\mathbb{C}) \longrightarrow M_2(\mathbb{C})$ be defined by $\varphi(A) = A^t $ (the transpose
map). It is known that $\varphi$  is a unital positive linear map, but it is not $2$-positive (and so $3$-positive). Take $A=\begin{bmatrix} 1&2\\ 2 &4\end{bmatrix}$ and $B=\begin{bmatrix} 1&0  \\ 0 &4\end{bmatrix}$.  Then we have  $\sigma(A)=\{0,5\}$ and $\sigma(B)=\{1,4\}$. So, $\inf_{\alpha \in \mathbb{C}}\|A-\alpha I_{2\times 2}\|= 2.5$ and $\inf_{\beta \in \mathbb{C}}\|B-\beta I_{2\times2}\|=1.5$. Moreover it is easy to check that $\alpha_0=2.5\in \Gamma_A , \beta_0=2.5 \in \Gamma_B$. Hence

\begin{eqnarray*}
\left|\varphi(AB)-\varphi(A)\varphi(B)\right|
&=&\begin{bmatrix} 6&0 \\0 &6\end{bmatrix}>\begin{bmatrix} 3.75&0  \\ 0 &3.75\end{bmatrix}\\ &=& \|\varphi(|A-\alpha_0 I_{2\times 2}|^2)\|^\frac{1}{2} [\varphi(|B-\beta_0 I_{2\times 2}|^2)]^\frac{1}{2}.
\end{eqnarray*}

Moreover, according to \cite{mathi}, if we take $\Psi: M_3(\mathbb{C}) \longrightarrow M_3(\mathbb{C})$ by $\Psi(A)=2Tr(A)I_{3\times 3}-A$, then $\Psi$ is a unital $2$-positive linear map but it is not $3$-positive. A similar calculation as above shows that inequality \eqref{covariance} of Theorem \ref{grussp} is not true for $A=\begin{bmatrix} 0&0&0 \\0&0&0\\ 1&0&0\end{bmatrix}$ and  $B=\begin{bmatrix} 0&1&0 \\0&0&0\\ 0&0&0\end{bmatrix}$.\\
\end{example}

Let us put $C_{S,T}(A) :=(A-S)^*(T-A)$. An easy computation shows that \begin{eqnarray*} {\rm Re}\ C_{S,T}(A)=\dfrac{1}{4}\big|T-S\big|^2 -\left|A-\dfrac{S+T}{2}\right|^2\end{eqnarray*} for every $A,T,S$ in a $C^*$-algebra.
So operator $C_{S,T}(A)$ is accretive if and only if \begin{equation}\label{eqacc} \left|A-\dfrac{S +T}{2} \right|^2 \leq \frac{1}{4} \big|T -S\big|^2.  \end{equation}
For simplicity, $C_{\alpha I,\beta I}(A)$ is denoted by $C_{\alpha ,\beta }(A),$ for $\alpha,\beta\in \mathbb{C}$.

Next we prove a lemma to give a Gr\"{u}ss type inequality for some classes of accretive operators.

\begin{lemma}\label{acc1}
Suppose that  $\mathcal{A}$ and $\mathcal{B}$ are two unital $C^*$-algebras, $A\in \mathcal{A}$ and $\alpha,\beta\in \mathbb{C}$ such that the operator $\Phi\big(C_{\alpha,\beta}(A)\big)$ is accretive. If $\Phi :\mathcal{A} \longrightarrow \mathcal{B}$ is a unital positive linear map, then
\begin{eqnarray*}
\Phi(A^*A)-\Phi(A)^*\Phi(A) &\leq& \frac{1}{4} \big|\beta -\alpha \big|^2 I- \Phi({\rm Re}\ C_{\alpha,\beta}(A))\\ &\leq&\frac{1}{4}\big|\beta -\alpha \big|^2 I.
\end{eqnarray*}
\end{lemma}
\begin{proof}
According to Lemma \ref{main1} we have
\begin{eqnarray*}
\Phi(A^*A)-\Phi(A)^*\Phi(A) \leq \Phi\left(\left|A-\dfrac{\alpha+\beta}{2}I\right|^2\right).
\end{eqnarray*}
Hence, by using inequality (\ref{eqacc}), we can write
\begin{eqnarray*}
\Phi(A^*A)-\Phi(A)^*\Phi(A) &\leq& \Phi\left(\left|A-\dfrac{\alpha+\beta}{2}I\right|^2\right)-\dfrac{1}{4}|\beta -\alpha|^2 I +\dfrac{1}{4}|\beta -\alpha|^2I \\ &=&\dfrac{1}{4}|\beta -\alpha|^2I - \Phi({\rm Re}\ C_{\alpha,\beta}(A))\\ &\leq&\frac{1}{4}|\beta -\alpha |^2 I.
\end{eqnarray*}
The last inequality follows from the accretivity of $\Phi \big(C_{\alpha,\beta}(A)\big).$
\end{proof}

Next, we use Lemma \ref{acc1} and Theorem \ref{grussp} to give the following Gr\"{u}ss type inequality.
\begin{corollary}\label{cor3}
Suppose that $A,B \in \mathcal{A}$ and $\Phi:\mathcal{A}\longrightarrow \mathcal{B}$ is a unital $n$-positive linear map for some $n\geq 3.$ Let $\alpha,\beta,\gamma,\Gamma \in \mathbb{C}$ and $\Phi\big(C_{\alpha,\beta}(A^*)\big),\Phi\big(C_{\gamma,\Gamma}(B)\big)$ be accretive operators. Then
\begin{eqnarray*}
\big|\Phi(AB)-\Phi(A)\Phi(B) \big| &\leq&\left\|\frac{1}{4} |\beta -\alpha |^2 I- \Phi\big({\rm Re}\ C_{\alpha,\beta}(A^*)\big)\right\|^\frac{1}{2}\\ && \cdot \left(\frac{1}{4} |\gamma -\Gamma|^2 I- \Phi\big({\rm Re}\ C_{\gamma,\Gamma}(B)\big)\right)^\frac{1}{2}\\ &\leq&\frac{1}{4}|\beta -\alpha | |\Gamma -\gamma | I.
\end{eqnarray*}
\end{corollary}

Here, we extend our work to give some Gr\"{u}ss type inequalities for the trace on a noncommutative probability space. In particular, the following inequalities are generalizations of inequality (\ref{rena}).

\begin{theorem}\label{trace1}
Let $(\mathcal{A},\tau)$ be a noncommutative probability space and $T \in \mathcal{A}$ be a density operator {\rm(}positive and $\tau (T)=1${\rm)}. Then the following inequalities hold:
\begin{enumerate}
\item{$|\tau(TAB)-\tau(TA)\tau(TB)|\leq \|A-\alpha I\|_4\|B-\beta I\|_4\|T\|_2$},
\item{$|\tau(TAB)-\tau(TA)\tau(TB)|\leq \|A-\alpha I\|_2\|B-\beta I\|_2 \|T\|$},
\item{$|\tau(TAB)-\tau(TA)\tau(TB)|\leq \|A-\alpha I\|\|B-\beta I\|$}
\end{enumerate}
for every $A,B \in \mathcal{A}$ and $\alpha,\beta \in \mathbb{C}$. In particular, by choosing $\alpha\in \Gamma_A$ and $\beta \in \Gamma_B$, (3) can be sharpened.
\end{theorem}

\begin{proof}
We define the map $\varphi :\mathcal{A} \longrightarrow \mathbb{C}$ by $\varphi (A)=\tau (TA).$ It is clear that $\varphi$ is a unital positive linear functional. So, by using Lemma \ref{main1}, we have
\begin{eqnarray}\label{vartrace1}
\tau (TAA^*) - \tau(TA)\tau(TA^*) \leq \tau(T|A^*-\bar{\alpha} I|^2)
\end{eqnarray}
for every $A\in \mathcal{A}$ and $\alpha \in \mathbb{C}$.
Hence, by using the Cauchy--Schwarz inequality and applying the fact that for all $C\in \mathcal{A}$, $\|C\|_p=\|C^*\|_p \ (p\geq 1)$, we get
\begin{eqnarray}
\tau (TAA^*) - \tau(TA)\tau(TA^*) &\leq& \left(\tau(|A^*-\bar{\alpha} I|^4)\right)^\frac{1}{2}\left(\tau(|T|^2)\right)^\frac{1}{2} \\ &=&\|A-\alpha I\|_4^2 \|T\|_2 \label{vartrace}
\end{eqnarray}
for every $\alpha \in \mathbb{C}$. Similarly
\begin{eqnarray}\label{vartrace2}
\tau (TB^*B) - \tau(TB^*)\tau(TB) \leq \|B-\beta I\|_4^2 \|T\|_2
\end{eqnarray}
for every $\beta\in\mathbb{C}.$
Now we define the map $(\cdot ,\cdot)_\tau :\mathcal{A} \times \mathcal{A} \longrightarrow \mathbb{C}$ by $(A,B)_\tau=\tau (TA^*B) - \tau(TA^*)\tau(TB).$
It is easy to check that $(\cdot,\cdot)$ is a usual semi-inner product on $\mathcal{A}$. Using the Cauchy--Schwarz inequality and inequalities (\ref{vartrace}) and (\ref{vartrace2}), we have
\begin{eqnarray*}
\big|\tau(TAB)-\tau(TA)\tau(TB)\big|&=& \big|(A^*,B)_\tau\big| \\ &\leq& (A^*,A^*)_\tau^\frac{1}{2}(B,B)_\tau^\frac{1}{2} \\ &\leq& \|A-\alpha I\|_4\|B-\beta I\|_4\|T\|_2
\end{eqnarray*}
for every $A,B \in \mathcal{A}$ and $\alpha,\beta \in \mathbb{C}.$ This proves the first inequality. \\ Note that in the right side of inequality (\ref{vartrace1}) (and similarly for inequality (\ref{vartrace2})) we can write $|A^*-\bar{\alpha} I| T|A^*-\bar{\alpha} I|\leq |A^*-\bar{\alpha} I|\| T\||A^*-\bar{\alpha} I|$. Hence $\tau (T|A^*-\bar{\alpha} I|^2) \leq \|A-\alpha I\|_2^2 \|T\|.$ Using this fact and applying the Cauchy--Schwarz inequality to $(\cdot,\cdot)_\tau$, we can get the second inequality.\\ Finally, we have
$T^\frac{1}{2}|A^*-\bar{\alpha} I|^2T^\frac{1}{2} \leq T^\frac{1}{2}\|A-\alpha I\|^2T^\frac{1}{2}.$ So $\tau( T|A^*-\bar{\alpha} I|^2) \leq \|A-\alpha I\|^2\tau(T).$ Applying this inequality to the right side of (\ref{vartrace1}) (and similarly for (\ref{vartrace2})) and using the Cauchy--Schwarz inequality we obtain the third inequality.
\end{proof}
\begin{corollary}
Let the assumptions of Theorem \ref{trace1} be held. Then
\begin{enumerate}
\item{$ |\tau(TAB)-\tau(TA)\tau(TB)|\leq \|A-\alpha I\|_p\|B-\beta I\|_q\|T\|_r$} \ $(p,q \geq 4, r\geq 2)$,
\item{$|\tau(TAB)-\tau(TA)\tau(TB)|\leq \|A-\alpha I\|_p\|B-\beta I\|_q \|T\|$} \ {\rm (}$p,q \geq 2${\rm)}
\end{enumerate}
for every $A,B \in \mathcal{A}$ and $\alpha, \beta \in \mathbb{C}$.
\end{corollary}
\begin{proof}
It is sufficient to apply Theorem \ref{trace1} and to note the fact that $\|C\|_p\leq \|C\|_q$, whenever $C\in \mathcal{A}$ and $1 \leq p \leq q$.
\end{proof}
Now we give an example to show that the right sides of inequalities (1) and (2) in Theorem \ref{trace1} are not comparable.
\begin{example}
Let $\tau$ be the usual trace on $M_2(\mathbb{C})$ given by $\tau ([a_{ij}])=\dfrac{a_{11}+a_{22}}{2}$. First consider $A=\begin{bmatrix} 3 & 0.5 \\ 0.5 & 2 \end{bmatrix}, B= \begin{bmatrix} 1 & 2\\ 2 & 4 \end{bmatrix}$ and $T=\begin{bmatrix} 1 & -0.1 \\ -0.1 & 1 \end{bmatrix}$. In this case we have
\begin{equation*}
\|A\|_4=2.76 ,\ \|B\|_4=4.20 ,\ \|T\|_2=1.004,
\ \|A\|_2=2.59 ,\ \|B\|_2=3.53 ,\ \|T\|=1.10 .
\end{equation*}
Hence, we see that
\begin{eqnarray*}
\|A\|_4 \|B\|_4 \|T\|_2= 11.63 >10.05=\|A\|_2 \|B\|_2 \|T\|.
\end{eqnarray*}
On the other hand, if we take $A=\begin{bmatrix} 1.5 &2 \\ 2&5 \end{bmatrix}, B=\begin{bmatrix} 3 &2 \\ 2&4 \end{bmatrix}$ and $T=\begin{bmatrix} 1 &-1 \\ -1&2\end{bmatrix}$, we get
\begin{equation*}
\ \ \|A\|_4=4.96,\ \|B\|_4=4.68 ,\ \|T\|_2=1.87,
\ \|A\|_2=4.19,\ \|B\|_2=4.06,\ \|T\|=2.61 .
\end{equation*}
Therefore,
\begin{eqnarray*}
\|A\|_4 \|B\|_4 \|T\|_2= 43.40 <44.39=\|A\|_2 \|B\|_2 \|T\|.
\end{eqnarray*}
So, none of inequalities (1) and (2) can imply the other.\\
\end{example}

Let $\alpha, \beta \in\mathbb{C}$ and $\tau \big(C_{\alpha,\beta}(A)\big)$ be accretive. According to (\ref{eqacc}), this condition is equivalent to
\begin{eqnarray}\label{tracac}
\tau \left(\left|A -\frac{\alpha+\beta}{2}\right|\right)=\left\|A -\frac{\alpha+\beta}{2}\right\|_1 \leq \frac{1}{2} \left|\beta -\alpha \right|.
\end{eqnarray}
Clearly, this condition is weaker than accretivity of $C_{\alpha,\beta}(A)$. Moreover, the accretivity of $\tau \big(C_{\alpha,\beta}(A)\big)$ means that $ -\dfrac{\pi}{2}\leq{\rm Arg} \big(\tau \big(C_{\alpha,\beta}(A)\big)\big)\leq \dfrac{\pi}{2}$, where ${\rm Arg} (z)$ is the argument of the complex number $z$.\\
If $0\leq A\leq B$, then it is easy to check that $\tau (A^{2^n}) \leq \tau (B^{2^n})$ for every $n\in \mathbb{R}$.
Indeed, \begin{equation*}
 \tau (B^2-A^2)=\tau \big((B-A)(B+A)\big) =\tau \big((B-A)^{\frac{1}{2}}(B+A)(B-A)^{\frac{1}{2}}\big) \geq 0.
\end{equation*}
 Hence, by using this fact and applying inequality (\ref{tracac}), we get
\begin{eqnarray}\label{acc4}
\left\|A -\frac{\alpha+\beta}{2}\right\|_4 \leq \frac{1}{2} \left|\beta -\alpha \right|.
\end{eqnarray}
Next corollary gives a Gr\"{u}ss type inequality for some accretive operators for the trace on a noncommutative probability space.

\begin{corollary}
Let $(\mathcal{A},\tau)$ be a noncommutative probability space and $T \in \mathcal{A}$ be a density operator. Let $\alpha,\beta,\zeta,\xi \in \mathbb{C}$ and $-\dfrac{\pi}{2}\leq {\rm Arg} \big(\tau \big(C_{\alpha,\beta}(A)\big)\big)\leq\dfrac{\pi}{2}$ and $-\dfrac{\pi}{2}\leq {\rm Arg} \big(\tau \big(C_{\zeta,\xi}(A)\big)\big)\leq\dfrac{\pi}{2}$. Then
\begin{eqnarray*}
|\tau(TAB)-\tau(TA)\tau(TB)|\leq \frac{1}{4}|\beta-\alpha| |\xi-\zeta| \|T\|_2.
\end{eqnarray*}
\end{corollary}
\begin{proof}
Using inequalities (\ref{vartrace}) and (\ref{acc4}), we get
\begin{eqnarray*}
\tau(TAA^*)-\tau(TA)\tau(TA^*)&\leq& \left\|A-\dfrac{\alpha+\beta}{2}\right\|_4^2 \|T\|_2 \\ &\leq& \dfrac{1}{4}|\beta -\alpha|^2 \|T\|_2.
\end{eqnarray*}
Similarly, we can write
\begin{eqnarray*}
\tau(TB^*B)-\tau(TB^*)\tau(TB)\leq \dfrac{1}{4}|\xi -\zeta |^2\|T\|_2.
\end{eqnarray*}
Therefore, by using the Cauchy--Schwarz inequality for $(\cdot,\cdot)_\tau$ (defined in Theorem \ref{trace1}), we conclude that
\begin{eqnarray*}
|(A^*,B)_\tau| &=&|\tau(TAB)-\tau(TA)\tau(TB)| \\ &\leq& (A^*,A^*)_\tau^\frac{1}{2} (B,B)_\tau^\frac{1}{2}\\ &\leq& \frac{1}{4}|\beta-\alpha| |\xi-\zeta| \|T\|_2.
\end{eqnarray*}
\end{proof}
We can replace the unital condition for a positive linear map $\Phi$ in Theorem \ref{grussp} by another mild condition. To achieve our result, first we give the following lemma.

\begin{lemma}\label{nounit}
Let $\mathcal{A}$ and $\mathcal{B}$ be unital $C^*$-algebras. If $\Phi :\mathcal{A} \longrightarrow \mathcal{B}$ is a positive linear map such that $\Phi(I)$ is invertible, then
\begin{eqnarray}
\Phi(A^*A)-\Phi(A)^*\Phi(I)^{-1}\Phi(A) \leq \Phi(|A-\alpha I|^2)
\end{eqnarray}
for all $A\in\mathcal{A}$ and $\alpha \in \mathbb{C} .$

\end{lemma}
\begin{proof}
Since $\Phi(I)>0$, the map $\Psi :\mathcal{A} \longrightarrow \mathcal{B}$ given by $\Psi (A)= \Phi(I)^{-\frac{1}{2}}\Phi(A)\Phi(I)^{-\frac{1}{2}}$ is a unital positive linear map. Employing Lemma \ref{main1} we infer that
\begin{eqnarray*}
\Phi(I)^{-\frac{1}{2}}\Phi(A^*A)\Phi(I)^{-\frac{1}{2}}&-&\Phi(I)^{-\frac{1}{2}}\Phi(A)^*\Phi(I)^{-1}\Phi(A) \Phi(I)^{-\frac{1}{2}}\\ &=&\Psi(A^*A)-\Psi(A)^*\Psi(A) \\ &\leq& \Psi(|A-\alpha I|^2)\\ &=&\Phi(I)^{-\frac{1}{2}} \Phi(|A-\alpha I|^2) \Phi(I)^{-\frac{1}{2}}
\end{eqnarray*}
for all $\alpha \in \mathbb{C.}$ Therefore, we can write
\begin{eqnarray*}
 \Phi(A^*A) -\Phi(A)^*\Phi(I)^{-1}\Phi(A) \leq \Phi(|A-\alpha I|^2).
\end{eqnarray*}

\end{proof}
Next, we give another version of Theorem \ref{grussp} by weakening the unital condition.
\begin{proposition}
Let $\mathcal{A}$ and $\mathcal{B}$ be unital $C^*$-algebras. If $\Phi :\mathcal{A} \longrightarrow \mathcal{B}$ is a $n$-positive linear map for some $n\geq 3$ such that $\Phi(I)$ is invertible, then
\begin{eqnarray}
\big|\Phi(AB)-\Phi(A)\Phi(I)^{-1}\Phi(B) \big|\leq \big\|\Phi(|A^*-\alpha I|^2)\big\|^\frac{1}{2} \big( \Phi(|B-\beta I|)^2)\big)^\frac{1}{2}
\end{eqnarray}
for every $A,B\in\mathcal{A}$ and $\alpha, \beta \in \mathbb{C}.$
\end{proposition}
\begin{proof}
From inequality (\ref{matrixp2}) and the $3$-positivity of $\Phi$, we deduce that
\begin{eqnarray*}
\begin{bmatrix}  \Phi(A^*A) & \Phi(A^*B) & \Phi(A^*) & 0  \\ \Phi(B^*A)  & \Phi(B^*B) &\Phi (B^*) & 0 \\ \Phi(A) &\Phi(B) & \Phi(I) & 0 \\ 0 & 0 & 0 & \Phi(I)
\end{bmatrix} \geq 0.
\end{eqnarray*}
Hence, the positivity of the latter matrix ensures that
\begin{eqnarray*}
\begin{bmatrix}  \Phi(A^*A) & \Phi(A^*B) \\ \Phi(B^*A)  & \Phi(B^*B)
\end{bmatrix} \geq \begin{bmatrix}  \Phi(A)^* & 0 \\ \Phi(B)^* & 0\end{bmatrix} \begin{bmatrix} \Phi(I)^{-1} & 0\\ 0 & \Phi(I)^{-1}
\end{bmatrix} \begin{bmatrix}  \Phi(A) & \Phi(B) \\ 0  & 0\end{bmatrix} .
\end{eqnarray*}
Therefore,
\begin{eqnarray*}
\begin{bmatrix}  \Phi(A^*A)-\Phi(A)^*\Phi(I)^{-1}\Phi(A) & \Phi(A^*B)-\Phi(A)^*\Phi(I)^{-1}\Phi(B)  \\ \Phi(B^*A)-\Phi(B)^*\Phi(I)^{-1}\Phi(A)   & \Phi(B^*B)-\Phi(B)^*\Phi(I)^{-1}\Phi(B)
\end{bmatrix} \geq 0 .
\end{eqnarray*}
Using Lemma \ref{nounit} and applying a similar argument as in the proof of Theorem \ref{grussp}, we get
 \begin{eqnarray*}
 \big|\Phi(AB)-\Phi(A)\Phi(I)^{-1}\Phi(B) \big|\leq \big\|\Phi(|A^*-\alpha I|^2)\big\|^\frac{1}{2} \big( \Phi(|B-\beta I|^2)\big)^\frac{1}{2}.
\end{eqnarray*}
\end{proof}

\section{Gr\"{u}ss type inequalities in Hilbert $C^*$-modules}\label{modulesection}

In this section, we give some Gr\"{u}ss type inequalities in the setting of Hilbert $C^*$-modules. There is a standard way for making a Hilbert $C^*$-module from a semi-inner product module, so we may work in the realm of Hilbert $C^*$-modules; see \cite[page 3-4]{lan}.\\
The range of an operator $T$ is denoted by ${\rm ran}(T)$. Let $T\in \mathcal{L}(\mathcal{X})$ be a positive operator. For $x,y \in \mathcal{X}$ we define
 \begin{eqnarray*}
 [x,y]_T =\langle Tx,y \rangle.
 \end{eqnarray*}
 It is easy to verify that $[x,y]_T$ is an $\mathcal{A}$-valued semi-inner product on $\mathcal{X}$. This means that $[x,x]_T=0 \Rightarrow x=0$ does not hold in general.

\begin{lemma}\label{lemgrussmodule}
Let $K\in \mathcal{L}(\mathcal{X})$ be a projection and $x,u,v\in \mathcal{X}.$ If $u+v \in {\rm{ran}} (K)$, then
\begin{eqnarray*}
[x,x]_{(I-K)}=|x|^2-\langle Kx,x\rangle \leq \frac{1}{4} |v-u|^2 -{\rm Re} \langle x-u,v-x \rangle.
\end{eqnarray*}
In particular, if ${\rm Re} \langle x-u,v-x \rangle\geq 0,$ then
\begin{eqnarray*}
|x|^2-\langle Kx,x\rangle \leq \frac{1}{4} |v-u|^2.
\end{eqnarray*}
\end{lemma}
\begin{proof}
First we show that
\begin{eqnarray*}\label{pos}
|x|^2-\langle Kx,x\rangle\leq |x-w|^2
\end{eqnarray*}
for every $w \in {\rm ran}(K)$. Indeed
\begin{eqnarray*}
|x-w|^2&\geq& |x-w|^2-|Kx-w|^2 \\ &=& \langle x-w,x-w\rangle -\langle Kx-w,Kx-w\rangle \\ &=& \langle x,x\rangle -\langle x,w\rangle - \langle w,x\rangle + \langle w,w\rangle- \langle Kx,Kx\rangle \\ &&+\langle Kx,w\rangle +\langle w,Kx\rangle -\langle w,w\rangle \\ &=& \langle x,x\rangle -\langle x,w\rangle - \langle w,x\rangle - \langle Kx,x\rangle \\ && +\langle x,Kw\rangle +\langle Kw,x\rangle \ \ ({\rm beacuse}\ w \in {\rm ran} (K)) \\ &=& |x|^2-\langle Kx,x\rangle.
\end{eqnarray*}
Due to $\dfrac{u+v}{2}\in {\rm ran}(K)$, we have
\begin{eqnarray*}
|x|^2-\langle Kx,x\rangle\leq \left|x-\dfrac{u+v}{2}\right|^2.
\end{eqnarray*}
Using the equality ${\rm Re}\langle x-u,v-x\rangle= \dfrac{1}{4}|v-u|^2 -\left|x -\dfrac{u+v}{2}\right|^2,$ we can write
\begin{eqnarray*}
|x|^2-\langle Kx,x\rangle &\leq& \left|x-\dfrac{u+v}{2}\right|^2 - \frac{1}{4}|v-u|^2 + \frac{1}{4}|v-u|^2 \\ &=& \frac{1}{4}|v-u|^2 - {\rm Re}\langle x-u,v-x\rangle.
\end{eqnarray*}
\end{proof}

Our main result of this section reads as follows.

\begin{theorem}\label{grussmodule}
Suppose that $\left(\mathcal{X},\langle \cdot ,\cdot \rangle\right )$ is a Hilbert module over a $C^*$-algebra $\mathcal{A}$ and $\{x,y,u,v,u',v'\}\subseteq \mathcal{X}.$
 Let \begin{eqnarray*} {\rm Re}\langle x-u,v-x\rangle \geq 0,\ {\rm Re} \langle y-u',v'-y\rangle \geq 0\end{eqnarray*} and
$K\in \mathcal{L}(\mathcal{X})$ be a projection such that $u+v, u'+v' \in {\rm{ran}}(K)$. Then
\begin{eqnarray*}
\big|\langle x ,y \rangle -\langle Kx,y\rangle \big| &\leq& \left\|\frac{1}{4}|v-u|^2 -{\rm Re}\langle x-u,v-x\rangle \right\|^\frac{1}{2} \\ && \cdot \left( \frac{1}{4} |v'-u'|^2 -{\rm Re}\langle y-u',v'-y\rangle \right)^\frac{1}{2}\\ &\leq& \frac{1}{4} \left\|v-u \right\| \left| v'-u' \right|.
\end{eqnarray*}
\end{theorem}
\begin{proof} The Cauchy--Schwarz inequality and Lemma \ref{lemgrussmodule} yield that
\begin{eqnarray*}
\big|\langle x ,y \rangle -\langle Kx,y\rangle \big| &=& \left|[x,y]_{(I-K)}\right| \\ & \leq& \left\|[x,x]_{(I-K)}\right\|^\frac{1}{2} \left|[y,y]_{(I-K)}\right|^\frac{1}{2} \\& \leq& \left\|\frac{1}{4}|v-u|^2 -{\rm Re} \langle x-u,v-x \rangle \right\|^\frac{1}{2}\\ && \cdot \left(\frac{1}{4}|v'-u'|^2 -{\rm Re} \langle y-u',v'-y \rangle\right)^\frac{1}{2}\\ \\ &\leq& \frac{1}{4} \left\||v-u |\right\| \left| v'-u' \right|\\ &=& \frac{1}{4} \left\|v-u \right\| \left| v'-u' \right|.
\end{eqnarray*}
\end{proof}

As a corollary, we have the following Gr\"{u}ss type inequality in framework of Hilbert $C^*$-modules (see \cite[Theorem 4.1]{il1} and \cite[Theorem 5.1]{il2}).
\begin{corollary}\label{lift}
Let $\mathcal{X}$ be a Hilbert module over a $C^*$-algebra $\mathcal{A}$ and $h\in\mathcal{X}$ be a lifted projection. If $a,A,b,B \in \mathcal{A}$ and $x,y\in \mathcal{X}$ such that the conditions
\begin{eqnarray}\label{halvin}
{\rm Re}\langle x-ha,hA-x\rangle \geq 0 , \ \ {\rm Re}\langle y-hb,hB-y\rangle \geq 0
\end{eqnarray}
hold, then
\begin{eqnarray*}
\big|\langle x,y\rangle -\langle x,h\rangle \langle h,y\rangle \big| &\leq& \left\|\frac{1}{4}|A-a|^2 -{\rm Re} \langle x-ha,hA-x \rangle \right\|^\frac{1}{2}\\ && \cdot \left(\frac{1}{4}|B-b|^2-{\rm Re} \langle y-hb,hB-y \rangle\right)^\frac{1}{2}
\\ &\leq& \frac{1}{4}\|A-a\||B-b|.
\end{eqnarray*}
\end{corollary}
 \begin{proof}
First note that for every $c \in \mathcal{A}$, we have $(h\otimes h)(hc)=h\langle h,hc\rangle =h\langle h,h\rangle c=hc$, whence $hc \in {\rm ran} ( h\otimes h)$. Thus, $\{ha,hA,hb,hB\} \subseteq {\rm{ran}} (h\otimes h)$. Moreover, since $\|h\|=1$, we have $|h(B-A)|^2 \leq |B-A|^2$. Employing Theorem \ref{grussmodule} we get
\begin{eqnarray*}
\big|\langle x,y\rangle -\langle x,h\rangle \langle h,y\rangle \big| &=& \big|\langle x,y\rangle -\langle (h\otimes h) x,y\rangle\big|\\ &\leq& \left\|\frac{1}{4}|h(A-a)|^2 -{\rm Re} \langle x-ha,hA-x \rangle \right\|^\frac{1}{2}\\ && \cdot \left(\frac{1}{4}|h(B-b)|^2-{\rm Re} \langle y-hb,hB-y \rangle\right)^\frac{1}{2} \\ &\leq& \left\|\frac{1}{4}|A-a|^2 -{\rm Re} \langle x-ha,hA-x \rangle \right\|^\frac{1}{2}\\ && \cdot \left(\frac{1}{4}|B-b|^2-{\rm Re} \langle y-hb,hB-y \rangle\right)^\frac{1}{2}
\\ &\leq& \frac{1}{4}\|A-a\||B-b|.
\end{eqnarray*}
\end{proof}
Similarly, in Corollary \ref{lift} we can replace $A,a,b,B \in \mathcal{A}$ with $\alpha ,\beta ,\gamma, \lambda \in \mathbb{C}$. \\
Another application of Theorem \ref{grussmodule} yields a Gr\"{u}ss inequality for a Hilbert space (as a Hilbert $\mathbb{C}$-module). It is proved  by Dragomir \cite[Theorem 3.1]{dra1}.
\begin{corollary}
Let $(\mathcal{H},\langle \cdot,\cdot \rangle)$ be a Hilbert space and $e\in \mathcal{H}$ be a unit vector. If $\alpha,\beta,\gamma$ and $\Gamma$ are real or complex numbers and $x,y$ are vectors in $\mathcal{H}$ such that ${\rm Re}\langle \beta e -x ,x-\alpha e\rangle \geq 0$ and ${\rm Re}\langle\Gamma e -y,y-\gamma e\rangle \geq 0$, then \begin{eqnarray*}
 \big|\langle x,y\rangle -\langle x,e\rangle \langle e,y\rangle \big| &\leq& \frac{1}{4}|\beta -\alpha | |\Gamma - \gamma| \\ && \ - \big({\rm Re}\langle \beta e -x ,x-\alpha e\rangle \big)^\frac{1}{2} \big({\rm Re}\langle\Gamma e -y,y-\gamma e\rangle\big)^\frac{1}{2}.
\end{eqnarray*}
\end{corollary}
\begin{proof}
Clearly, $e\otimes e$ is a rank one projection in $\mathbb{B}(\mathcal{H})$. Similar to the proof of Corollary \ref{lift}, $\{\alpha e,\beta e,\gamma e,\Gamma e\} \subseteq {\rm ran} ( e\otimes e).$
Based on Theorem \ref{grussmodule} we have
\begin{eqnarray*}
\big|\langle x,y\rangle -\langle x,e\rangle \langle e,y\rangle \big|^2 & =&\big|\langle x,y\rangle -\langle (e\otimes e) x,y\rangle \big|^2 \\ &\leq &\left[\frac{1}{4}\|(\beta -\alpha)e\|^2 -([{\rm Re} \langle \beta e -x ,x-\alpha e\rangle]^\frac{1}{2})^2\right] \\ & & \cdot \left[ \frac{1}{4}|(\Gamma - \gamma)e|^2 -([{\rm Re}\langle\Gamma e -y,y-\gamma e\rangle]^\frac{1}{2} )^2\right] \\ &\leq& \left[\frac{1}{4}|\beta -\alpha|^2 -([{\rm Re} \langle \beta e -x ,x-\alpha e\rangle]^\frac{1}{2})^2\right] \\ & & \cdot \left[ \frac{1}{4}|\Gamma - \gamma|^2 -([{\rm Re}\langle\Gamma e -y,y-\gamma e\rangle]^\frac{1}{2} )^2\right]\\ &\leq & \frac{1}{4}|\beta -\alpha | |\Gamma - \gamma| \\ &&\ -  \left([{\rm Re}\langle \beta e -x ,x-\alpha e\rangle ]^\frac{1}{2} [{\rm Re}\langle\Gamma e -y,y-\gamma e\rangle]^\frac{1}{2} \right) ^2 .
\end{eqnarray*}
Note that the last inequality follows from the elementary inequality
\begin{eqnarray*}
(m^2-n^2)(p^2-q^2)\leq (mp-nq)^2
\end{eqnarray*}
 for positive real numbers.
\end{proof}

Now we give a Gr\"{u}ss type inequality for accretive operators.

\begin{proposition}\label{accf}
Let $\mathcal{X}$ be a Hilbert module over a $C^*$-algebra $\mathcal{A}$. Let $T\in \mathcal{L}(\mathcal{X})$ be a self-adjoint operator. If $A,B \in \mathcal{Z}(\mathcal{A})$ and operator $C_{A,B}(T)=(T-R_A)^*(R_B-T)$ is accretive, then
\begin{eqnarray*}
\langle T^2 h,h \rangle - \langle Th,h\rangle^2 \leq \frac{1}{4}|B-A|^2- \langle{\rm Re}\ C_{A,B}(T) h, h \rangle \leq \frac{1}{4}|B-A|^2
\end{eqnarray*}
for all lifted projections $h \in\mathcal{X}$.

\end{proposition}

\begin{proof}
First note that $|h(B-A)|^2 \leq |B-A|^2$. Moreover, for every lifted projection $h$, $\{hA,hB\}\subseteq {\rm ran}(h\otimes h)$. Employing Lemma \ref{lemgrussmodule} for $x=Th, u=hA, v=hB$ and $K=h \otimes h$ we get the result.
\end{proof}
\begin{corollary}
Let $T \in \mathbb{B}(\mathcal{H})$ be a self-adjoint operator and $m,M $ be positive numbers. If $mI\leq T \leq MI$, then
\begin{eqnarray*}
\langle T^2 x,x \rangle - \langle Tx,x\rangle^2 \leq \frac{1}{4}(M-m)^2- \langle(T-mI)(MI-T)x,x\rangle,
\end{eqnarray*}
for every unit vector $x\in \mathcal{H}$.
\end{corollary}
\begin{proof}
Clearly, each unit vector is a lifted projection in $\mathcal{H}$. A use of Proposition \ref{accf} for positive (and so accretive) operator $C_{m,M}(T)$  completes the proof.
\end{proof}

\end{document}